\documentclass{article}
\usepackage{graphicx,epstopdf}
\usepackage[centertags]{amsmath}
\usepackage{amsfonts,amssymb,amsthm,mathrsfs,newlfont,url,textcomp,geometry}
\usepackage[usenames, dvipsnames]{color}
\usepackage{soul,multicol,units,afterpage,mathtools,caption,multirow,array,booktabs,adjustbox,listings}

\usepackage{fourier}

\DeclareMathOperator{\sinc}{sinc}
\newtheorem{theorem}{Theorem}

\title{Asymptotic evaluation of the Sinc transform of entire exponential type function resulting to exact polynomial asymptotic behavior}
\author{Nathalie Liezel R. Rojas and Eric A. Galapon \\
{\normalsize Theoretical Physics Group,  National Institute of Physics } \\
{\normalsize University of the Philippines Diliman } \\
{\normalsize Quezon City, 1101 Philippines}}
\date{\today}

\begin{document}
\maketitle

\begin{abstract}
We consider the asymptotic evaluation of the integral transform $\int_0^\infty f(x) \, \sin^n(\lambda x)/x^n \,\text{d} x$ of an exponential type function $f(x)$ of type $\tau>0$, for large values of the parameter $\lambda$, where $n$ is a positive integer. We refer to this integral as the Sinc transform. Under the condition that $f(x)$ is even with respect to $x$, we derive a terminating asymptotic expansion of the Sinc transform which behave as a polynomial in positive powers of $\lambda$ as $\lambda$ grows large provided that the conditions $\lambda > \tau/2$ for even $n$ and $\lambda>\tau$ for odd n are satisfied.
\end{abstract}

\hspace{1em} {\small Keywords: Poincar{\'e} asymptotic expansion, integral transform, Sinc function, entire exponential type function}

\hspace{1em} {\small MSC Classification: 30D15, 41A60, 44A15}

\section{Introduction}
An integral transform given by $F(\lambda) = \int_a^b \, f(x) \, K (\lambda x) \text{d}x$ converts a function $f(x)$ over an interval $a \leq x \leq b$ into a new function $F(\lambda )$ where $K (\lambda x)$ is the kernel of transformation.
The importance of the integral transform lies in its ability to convert a complex mathematical problem into a relatively simpler one, which can be more easily solved \cite{debnath2016integral}. For many common transforms where $a=0$ and $b=\infty$, approximation methods are often necessary for the evaluation of these integrals for large values of the parameter $\lambda$ because some do not have exact solutions that can be expressed in terms of standard functions, but some methods result in approximations where the leading term may contain singularities or discontinuities \cite{paulsen2013asymptotic}. 
One approach to overcome these limitations is the use of the Poincar{\'e} asymptotic expansions (PAE). PAE offers a powerful method for approximating solutions to functions such as integral transforms, particularly when the parameter becomes large.
If there exists a power series representation of an integral transform which may converge or diverge, the  Poincar{\'e} asymptotic expansion for arbitrarily large $\lambda$ of this integral transform is given by $F(\lambda ) \sim \sum_{n=0}^\infty \, a_n \, \lambda^{-n}$
where $a_n$ are the coefficients of the expansion
\cite{wong2001asymptotic}.
For large parameter values, the Poincar{\'e} asymptotic expansion may yield a divergent series.
In general, the series does not terminate and a truncated series of the PAE remains only as an approximation of the integral transform that captures its behavior as the parameter grows large.

Recently, an asymptotic analysis of the Hankel transform, $\int_0^\infty f(x) J_{\nu} (\lambda x) \text{d}x$, of an entire exponential type function, $f(x)$, for arbitrary large $\lambda$ was done in \cite{rojas2025terminating}. Terminating asymptotic expansions were derived in which the expansions exhibit behavior that are similar to the PAE. That is, $\int_0^\infty f(x) J_{\nu} (\lambda x) \text{d}x = \sum_{k=0}^N a_k \lambda^{-k}$, $N<\infty$ for sufficiently large $\lambda$. This result was obtained by considering the integrand, $f(x) J_{\nu} (\lambda x)$, to be odd with respect to $x$ and performing a shifting of contour integration in the complex plane.
It was shown that the exponentially small trailing terms in the asymptotic expansion vanish when the asymptotic parameter is greater than the type of the exponential type function leading to an exact evaluation of the Hankel integral transform rather than an approximation of the integral. 
These findings raise now the question of how a terminating Poincar{\'e} asymptotic expansion can be obtained, as such expansions are not common in asymptotic analysis. This observation provides motivation to further investigate the conditions under which terminating expansions can arise.

In this paper, we aim to asymptotically evaluate the integral transform given by
\begin{equation}\label{eqn:IntTrans}
    I (\lambda) = \int_{0}^\infty  \frac{f(x) \, \sin^n (\lambda x) }{x^n} \text{d}x,
\end{equation}
where $f(x)$ is an entire function of exponential type $\tau>0$ and is even in $x$. We refer to this integral as the Sinc transform, where $ \sin (x)/ x $, or $\sinc (x) $ is the sine cardinal function \cite{mathworldSinc}.  
An entire function $f(x)$ with a complex extension $f(z)$ in the complex plane is of exponential type $\tau$ if it satisfies the inequality $|f(z)| \leq   e^{(\tau + \epsilon) |z|}$ for sufficiently large $|z|$ and for any arbitrarily small $\epsilon>0$ \cite{galapon2016internal, levin1996lectures,boas1954entire}. 
A specific example of the Sinc transform for $n=2$ in equation \eqref{eqn:IntTrans} arises in the calculation of the transition probability to any state using first-order perturbation theory in quantum mechanics. In the time-dependent integration of the oscillatory terms in the continuum limit, the function $f(x)$ in the Sinc transform represents the product of the coupling strength and density of states, which are typically treated as constants. When this conditions holds, the integration leads to the well-known Fermi's Golden rule \cite{griffiths2018introduction,sakurai2020modern}. However, we shall consider this problem elsewhere, as it lies beyond the scope of the present paper.

Here, we wish to show that, under certain conditions, the Sinc transform has an exact polynomial behavior of order $n$ as the parameter $\lambda$ grows large. We do this by writing the integral in equation \eqref{eqn:IntTrans} as
\begin{align}\label{eqn:IntInf}
    \int_{0}^\infty  \frac{f(x) \,  \sin^n (\lambda x) }{x^n} \text{d}x \, = \, \frac{1}{2} \, \int_{-\infty}^\infty \frac{f(x) \,   \sin^n (\lambda x) }{x^n} \text{d}x.
\end{align}
We shall evaluate the right-hand side by performing a contour integration in the complex plane in which the paths bypass the origin. While maintaining the even parity of $f(x)$, we recognize two cases; when $n$ is even, and when $n$ is odd. 
We find that the integral $I(\lambda)$ results to a terminating asymptotic expansion which behaves as a polynomial provided that $\lambda>\tau/2$ for even values of $n$ and $\lambda>\tau$ for odd values of $n$, which deviates from the typical large-order behavior of asymptotic expansions as seen in the Poincar{\'e} asymptotic expansion.

The rest of the paper is organized as follows. In Section 2, we evaluate the integral in equation \eqref{eqn:IntTrans} for the first case when $n$ is even, and show that the obtained expansion is terminating with polynomial behavior. In Section 3, we again evaluate the integral in equation \eqref{eqn:IntTrans} for the second case when $n$ is odd and also show that we obtain a terminating expansion that exhibits polynomial behavior. We shall also consider examples of the Sinc transform in equation \eqref{eqn:IntTrans} for specific values of $n$, in both Section 2 and 3.

\section{Case 1}

\begin{theorem}
    Let $f(x)$ be an entire function of exponential type $\tau>0$ satisfying $f(-x)=f(x)$. For positive integer $m$ and $\lambda > \tau/2$, we have
    \begin{align}
        \int_{0}^{\infty} \, \frac{f(x) \, \sin^{2m} (\lambda x)}{x^{2m}} \, \text{d}x = \, \frac{(2m)!}{2^{2m} \, (m!)^2} \, \;\backslash\!\!\!\!\!\!\!\backslash\!\!\!\!\!\!\int_{0}^{\infty} \, \frac{f(x)}{x^{2m}} \, \text{d}x   \, - & \,  \frac{\pi \, (-1)^m (2m)!}{2^{2m}} \,  \sum_{l=0}^{m-1} \, C_{l,m}  \, \lambda^{2l+1},
\end{align}
where     
$\,\,\, \;\backslash\!\!\!\!\!\!\!\backslash\!\!\!\!\!\!\int_{0}^{\infty} \, \frac{f(x)}{x^{2m}} \, \text{d}x$ is the finite part  of the divergent integral $\int_{0}^{\infty} \, \frac{f(x)}{x^{2m}} \, \text{d}x$, and 
\begin{align}
    C_{l,m} = \frac{ (-1)^l \, 2^{2l+1} \, f^{(2m-2l-2)} (0) \,}{(2l+1)! \, (2m-2l-2)!} \, \sum_{k=0}^{m-1} \, \frac{(-1)^k \, (m-k)^{2l+1}}{k! \, (2m-k)!} \, .
\end{align}
\end{theorem}

\begin{proof}
Let 
\begin{equation}\label{eqn:IntSinc2m}
    I_1 (\lambda) = \int_{0}^{\infty} \, \frac{f(x) \, \sin^{2m} (\lambda x)}{x^{2m}} \, \text{d}x.
\end{equation}
By hypothesis, $f(-x)=f(x)$ so that the integrand of $I_1 (\lambda)$ is even in $x$ which allows us to write
\begin{align}\label{eqn:PV1}
    \int_{0}^{\infty} \, \frac{f(x) \, \sin^{2m} (\lambda x)}{x^{2m}} \, \text{d}x \, = \, \frac{1}{2} \, \int_{-\infty}^{\infty} \, \frac{f(x) \, \sin^{2m} (\lambda x)}{x^{2m}} \, \text{d}x.
\end{align}
We now proceed with performing a contour integration in the complex plane. The resulting integral on the right hand side of equation \eqref{eqn:PV1} can be written as the average of two contour integrals \cite{galapon2016cauchy},
\begin{equation}\label{eqn:Gamma+-}
    \int_{-\infty}^{\infty} \, \frac{f(x) \, \sin^{2m} (\lambda x)}{x^{2m}} \, \text{d}x \, =  \, \frac{1}{2} \left( \int_{\Gamma^+} \, \frac{f(z) \, \sin^{2m} (\lambda z)}{z^{2m}} \, \text{d}z + \, \int_{\Gamma^-} \, \frac{f(z) \, \sin^{2m} (\lambda z)}{z^{2m}} \, \text{d}z \right),
\end{equation}
where the paths $\Gamma^+$ and $\Gamma^-$ bypass the origin in the upper and lower half-plane respectively, as shown in Figure \ref{fig:Figure1}.
\begin{figure}
\begin{center}
\setlength{\unitlength}{1cm}
\begin{picture}(12,8)
  \thicklines
  \put(1,4){\line(1,0){10}}    
  \put(6,3.25){\line(0,1){4}}    

  \put(2,7){\line(1,0){8}}      
  \put(2,7){\line(0,-1){3}}     
  \put(10,7){\line(0,-1){3}}    

  \qbezier(2,4.05)(6,4.8)(10,4.05)
  \qbezier(2,3.95)(6,3.2)(10,3.95)

  \put(5,4.65){$\Gamma_+$}
  \put(6.45,3){$\Gamma_-$}

  \put(8,4.65){$C^+$}
  \put(8,3.15){$C^-$}

  \put(1.8,3.65){-$\overset{|}{L}$}
  \put(9.9,3.65){$\overset{|}{L}$}
  
  \put(6.25,6.65){$i \gamma$}

  \put(10.2,5.4){$L_1$}  
  \put(5,7.2){$L_2$}   
  \put(1.5,5.4){$L_3$}   

  \put(5.15,4.4){\vector(1,0){.25}}  
  \put(6.4,3.58){\vector(1,0){.25}}  
  \put(10,5.4){\vector(0,1){0.25}}  
  \put(5.3,7){\vector(-1,0){0.5}}  
  \put(2,5.4){\vector(0,-1){.25}}  

\end{picture}
\end{center}
\vspace{-3.5cm}
\caption{The contour $C^+$ consists of the path $ \Gamma_+ $, $ L_1 $, $ L_2 $ and $ L_3$, while the contour $C^-$ consists of $ \Gamma_- $, $ L_1 $, $ L_2 $ and $ L_3$.}
\label{fig:Figure1}
\end{figure}
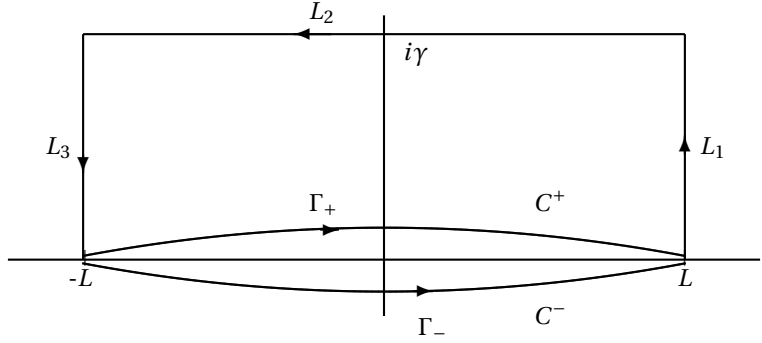
Next, we write the sine function in terms of its complex exponential function form $\sin (\lambda z) \, = (e^{i \lambda z} - e^{- i \lambda z})/2 i$ to obtain
\begin{align}\label{eqn:SineSumBinExp}
    \sin^{2m} (\lambda z) \, = \, & \, \frac{(-1)^m \, (2m)!}{2^{2m}} \, \sum_{k=0}^{2m} \,\frac{(-1)^k}{k! \, (2m-k)!} \ e^{2 i \lambda z (m-k)}.
\end{align}
This result follows from applying the binomial expansion to the expression obtained after using the complex exponential form.
We insert equation \eqref{eqn:SineSumBinExp} back into the right-hand side of equation \eqref{eqn:Gamma+-}, and an interchange on the order of summation and integration is done. Rearranging the terms, we get
\begin{align}\label{eqn:Thm1Exp0}
    \frac{1}{2} \left( \int_{\Gamma^+} \, \frac{f(z) \, \sin^{2m} (\lambda z)}{z^{2m}} \, \text{d}z + \, \int_{\Gamma^-} \, \frac{f(z) \, \sin^{2m} (\lambda z)}{z^{2m}} \, \text{d}z \right) \, = \, & \, \frac{(-1)^m}{2^{2m+1}} \, \sum_{k=0}^{2m} \,\frac{(2m)! \, (-1)^k}{k! \, (2m-k)!} \,  \, \text{d}z & \nonumber \\
     &\hspace{-10mm} \times \left( \, \int_{\Gamma^+} \, \frac{f(z) \, e^{2 i \lambda z (m-k)}}{z^{2m}} \, \text{d}z + \, \int_{\Gamma^-} \, \frac{f(z) \, e^{2 i \lambda z (m-k)}}{z^{2m}} \, \text{d}z \right).
\end{align}
Isolating the $k=m$-th term in the summation of equation \eqref{eqn:Thm1Exp0}, we result to three set of equations. The first set contains a sum from $k=0,\dots,m-1$ while second is the term containing the $k=m$-th term and the third contains a sum from $k=m+1,\dots,2m$. In these equations, we observe that the powers of the exponential function are positive for values of $k=0, \dots, m-1$. However, the powers become negative for values of $k=m+1, \dots, 2m $. This leads now to
\begin{align}\label{eqn:IntExpandGen}
    I_1 (\lambda) = & \, \frac{(2m)!}{2^{2m+1} \, (m!)^2} \, \left( \frac{1}{2} \int_{\Gamma^+} \, \frac{f(z)}{z^{2m}}  \, \text{d}z + \frac{1}{2} \int_{\Gamma^-} \, \frac{f(z)}{z^{2m}}  \, \text{d}z \right) \,
    \, \nonumber \\
    & \, \quad \, + \,  \frac{(-1)^m (2m)!}{2^{2m+2}} \, \sum_{k=0}^{m-1} \,\frac{(-1)^k}{k! \, (2m-k)!} \left( \int_{\Gamma^+}  \, \frac{f(z) \, e^{2 i \lambda z (m-k)} }{z^{2m}} \, \text{d}z  + \, \int_{\Gamma^-}  \, \frac{f(z) \, e^{2 i \lambda z (m-k)} }{z^{2m}} \, \text{d}z \right)\, 
     \, \nonumber \\
    & \, \quad \, + \frac{(-1)^m (2m)!}{2^{2m+2}} \, \sum_{k=m+1}^{2m} \,\frac{(-1)^k}{k! \, (2m-k)!} \, \left( \int_{\Gamma^+}  \, \frac{f(z) \, e^{-2 i \lambda z (k-m)} }{z^{2m}} \, \text{d}z \,  + \, \int_{\Gamma^-}  \, \frac{f(z) \, e^{-2 i \lambda z (k-m)} }{z^{2m}} \, \text{d}z \right).
\end{align}

We now solve the terms enclosed in parenthesis on the first term of equation \eqref{eqn:IntExpandGen}.
Upon examination, we identify these terms as the Hadamard finite part \cite{galapon2016cauchy}
\begin{align} \label{eqn:FPI1}
 \left( \frac{1}{2} \int_{\Gamma^+} \, \frac{f(z)}{z^{2m}}  \, \text{d}z + \frac{1}{2} \int_{\Gamma^-} \, \frac{f(z)}{z^{2m}}  \, \text{d}z \right) = \;\backslash\!\!\!\!\!\!\!\backslash\!\!\!\!\!\!\int_{-\infty}^{\infty} \, \frac{f(x)}{x^{2m}} \, \text{d}x,
\end{align}
of the following divergent integral 
\begin{align}\label{eqn:DivIntGen}
    \int_{-\infty}^{\infty} \, \frac{f(x)}{x^{2m}} \, \text{d}x.
\end{align}
The Hadamard finite part is determined by systematically eliminating the singular contribution arising from the non-integrable singularity at $x=0$. 
This process involves decomposing the integral in equation \eqref{eqn:DivIntGen} into two separate integrals, one over the interval $(-\infty,0)$ and the other over $(0,\infty)$. By treating these regions independently, the divergent behavior at $x=0$ can be symmetrically isolated and removed.
The singularity at $x=0$ is then excluded by modifying the integration limits to the intervals $(-\infty,-\epsilon)$ and  $(\epsilon,\infty)$ for some positive $\epsilon$. The resulting integral is expressed in the form
\begin{equation} \label{eqn:FPIDecompGen}
    \int_{-\infty}^{-\epsilon} \, \frac{f(x)}{x^{2m}} \, \text{d}x + \int_{\epsilon}^{\infty} \, \frac{f(x)}{x^{2m}} \, \text{d}x = \, \mathfrak{D}_{\epsilon} \,  + \, \mathfrak{C}_{\epsilon},
\end{equation}
where $\mathfrak{D}_{\epsilon}$ 
consists of the algebraic powers of $\epsilon$ and $\ln(\epsilon)$, which diverge in the limit as $\epsilon \to 0$ whereas $\mathfrak{C}_{\epsilon}$ consists of the terms that converge in the same limit \cite{galapon2017problem}. Then we have
\begin{equation}
    \, \;\backslash\!\!\!\!\!\!\!\backslash\!\!\!\!\!\!\int_{-\infty}^{\infty} \, \frac{f(x)}{x^{2m}} \, \text{d}x  \, = \, \lim_{\epsilon \to 0} \, \mathfrak{C}_{\epsilon}.
\end{equation}

The computation of the Hadamard finite part is simplified further by transforming the divergent integral from an integral over $(-\infty,\infty)$ to one over $(0,\infty)$. By substituting $x \to -x$ in the first term on the left-hand side of equation \eqref{eqn:FPIDecompGen} where we note that $f(-x)=f(x)$, we obtain
\begin{equation}\label{eqn:FPIDecompEpsilonGen}
    \int_{-\infty}^{-\epsilon} \, \frac{f(x)}{x^{2m}} \, \text{d}x + \int_{\epsilon}^{\infty} \, \frac{f(x)}{x^{2m}} \, \text{d}x = 2 \,    \int_{\epsilon}^{\infty} \, \frac{f(x)}{x^{2m}} \, \text{d}x.
\end{equation}
The integral on the right-hand side of equation \eqref{eqn:FPIDecompEpsilonGen} corresponds now to the divergent integral,
\begin{align}
    \int_{0}^{\infty} \, \frac{f(x) }{ x^{2m}} \, \text{d}x,
\end{align}
except that the singularity at $x=0$ has been removed by replacing the lower limit of the integration by some positive $\epsilon$. Further, we can write $\int_{\epsilon}^{\infty} \, \frac{f(x) }{ x^{2m}} \, \text{d}x = D_{\epsilon} + C_{\epsilon}$
where $D_{\epsilon}$ represents the terms that diverge in the limit as $\epsilon \to 0$ while $C_{\epsilon}$ represents the terms that converge in the same limit \cite{galapon2017problem}. 
Therefore, the Hadamard finite part of the divergent integral is expressed as
\begin{equation}
\;\backslash\!\!\!\!\!\!\!\backslash\!\!\!\!\!\!\int_{0}^{\infty} \, \frac{f(x)}{x^{2m}} \, \text{d}x = \lim_{\epsilon \to 0} \, C_{\epsilon}.
\end{equation}
It is clear that $\mathfrak{C}_{\epsilon} = 2 C_{\epsilon}$ so we finally obtain
\begin{equation}\label{eqn:FPIevenGen}
     \;\backslash\!\!\!\!\!\!\!\backslash\!\!\!\!\!\!\int_{-\infty}^{\infty} \, \frac{f(x)}{x^{2m}} \, \text{d}x= 2 \,\,  \;\backslash\!\!\!\!\!\!\!\backslash\!\!\!\!\!\!\int_{0}^{\infty} \, \frac{f(x)}{x^{2m}} \, \text{d}x .
\end{equation}

We now evaluate the integrals on the second term of equation \eqref{eqn:IntExpandGen}.
Here, the contour integration is on the upper half-plane. In Figure \ref{fig:Figure1}, we consider the contour $C^+$ where
\begin{align}
    \oint_{C^+} \frac{f(z) \, e^{2 i \lambda z (m-k)}}{z^{2m}} \, \text{d}z \, & \,  = \, \int_{\Gamma^+} \frac{f(z) \, e^{2 i \lambda z (m-k)}}{z^{2m}} \, \text{d}z \, + \int_{L_1} \, 
    \frac{f(L + i y) \, e^{2 i \lambda (L + i y) (m-k)}}{(L + i y)^{2m}}  \, \text{d}y \, \nonumber \\
    & \quad \, + \int_{L_2}  \frac{f(x + i \gamma) \, e^{2 i \lambda (x + i \gamma) (m-k)}}{(x + i \gamma)^{2m}} \, \text{d}x \, +\int_{L_3} \, 
    \frac{f(-L + i y) \, e^{2 i \lambda (-L + i y) (m-k)}}{(-L + i y)^{2m}}  \, \text{d}y 
    \nonumber \\
    & \, = 0.
\end{align}
This integral is equal to zero because the contour $C^+$ encloses no poles of the integrand. Next, we estimate the contributions along the paths $L_1, L_2$ and $L_3$. For $L_1$, we have
\begin{equation}
    \bigg|  \int_{0}^{\gamma} \, f(L + i y) \,
    \frac{e^{2 i \lambda (L + i y) (m-k)}}{(L + i y)^{2m}}  \, \text{d}y \bigg| \leq   \int_{0}^{\gamma} \, \big| f(L + i y) \big| \,\big|  \frac{e^{2 i \lambda (L + i y) (m-k)}}{(L + i y)^{2m}} \big| \, \text{d}y \leq \, e^{- 2 \lambda (m-k)} \, \int_{0}^{\gamma} \, \big| f(L + i y) \big| \, \text{d}y.
\end{equation}
Since $f(x)$ is an exponential type function and $f(x)$ vanishes as $|x|\to \infty$, the integral $\int_0^{\gamma} |f(L+iy)| \mbox{d}y$ vanishes as $L\rightarrow\infty$ \cite{boas1954entire}. Then the integral along the path-$L_1$ vanishes in the limit as $L\to\infty$, and the same case also happens for the path along $L_3$ as $L \to -\infty$ \cite{boas1954entire, schmeisser2007approximation}.
For the path along $L_2$, we consider its upper bound where
\begin{equation}\label{eqn:L2estimate}
    \bigg| \int_{L}^{-L} \, f(x + i \gamma) \, \frac{e^{2 i \lambda (x + i \gamma) (m-k)}}{(x + i \gamma)^{2m}} \, \text{d}x \bigg| \,   \, \leq \, \int_{-\infty}^{\infty} \, \big|  f(x + i \gamma) \big| \, \big| \frac{e^{2 i \lambda (x + i \gamma) (m-k)}}{(x + i \gamma)^{2m}} \big|  \, \text{d}x \, \leq \, e^{-2 \gamma  \lambda (m-k)} \, \, \int_{-\infty}^{\infty} \, \frac{\big|  f(x + i \gamma) \, \big|}{(x^2  + \gamma^2)^{m}}  \, \text{d}x.
\end{equation}
On the right hand side of equation \eqref{eqn:L2estimate}, the integrand is of exponential type, so we apply the inequality
$\int_{-\infty}^{\infty} |g(x+iy)|^p \mbox{d}x \leq e^{p \tau |y| } \int_{-\infty}^{\infty}   |g(x)|^p \mbox{d}x$ where $g(z)$ is an exponential function of type $\tau$ \cite{boas1954entire}. For $p=1$, we obtain 
\begin{equation}
    e^{-2 \gamma  \lambda (m-k)} \, \int_{-\infty}^{\infty} \, \frac{\big|  f(x + i \gamma) \, \big|}{(x^2  + \gamma^2)^{m}}  \, \text{d}x \,  \leq \, e^{-2 \gamma  \lambda (m-k)} \, e^{\gamma \tau} \, \int_{-\infty}^{\infty} \, \frac{\big|  f(x) \, \big|}{(x^2 +   \gamma^2)^{m}}  \, \text{d}x \, \leq \, M  \, e^{- \gamma (2 \lambda (m-k) - \tau)},
\end{equation}
for $M<\infty$. When $\lambda > \tau/2(m-k)$ for $k=0,\dots,m-1$, the bound on the integral along $L_2$ can be made arbitrarily small by extending $L_2$ further upward, corresponding to an arbitrarily large $\gamma'>\gamma$. This implies that the integral along $L_2$ vanishes provided that the condition 
\begin{align}\label{eqn:condition1}
\lambda > \frac{\tau}{2(m-k)}, \,\,\,\, \mbox{for} \,\,\,\, k=0,\dots,m-1,
\end{align}
is satisfied. The integral along $L_2$ vanishes because $f(z)$ is an entire exponential function of type $\tau$.
With this condition, we finally get 
\begin{align}
    \int_{\Gamma^+} \, \frac{f(z) \, e^{2 i \lambda z (m-k)}}{z^{2m}} \, \text{d}z \, = \, 0.
\end{align}

We now consider the contour $C^-$ in the upper half plane which includes now the path along $\Gamma^-$ in Figure 1. We have
\begin{align}
    \oint_{C^-} \, \frac{f(z) \, e^{2 i \lambda z (m-k)}}{z^{2m}} \, \text{d}z \, & = \, \int_{\Gamma^-} \frac{f(z) \, e^{2 i \lambda z (m-k)}}{z^{2m}} \, \text{d}z \, + \int_{L_1} \, f(L + i y) \,
    \frac{e^{2 i \lambda (L + i y) (m-k)}}{(L + i y)^{2m}}  \, \text{d}y \, \nonumber \\
    & \quad \, + \int_{L_2} f(x + i \gamma) \, \frac{e^{2 i \lambda (x + i \gamma) (m-k)}}{(x + i y)^{2m}} \, \text{d}x \, +\int_{L_3} \, f(-L + i y) \,
    \frac{e^{2 i \lambda (-L + i y) (m-k)}}{(-L + i y)^{2m}}  \, \text{d}y \, \nonumber \\
    & \, = \, 2 \pi i \, \text{Res} \left[\frac{f(z) \, e^{2 i \lambda z (m-k)} }{z^{2m}},z=0 \right].
\end{align}
The integration encloses the pole of the integrand at $z=0$, and thus, it evaluates to the sum of the residues at this pole.
Using the same estimate for the paths along $L_1, L_2$ and $L_3$ where the contributions from these paths vanish as $L\to \pm \infty$ \cite{boas1954entire,schmeisser2007approximation}, we obtain
\begin{equation}\label{eqn:Gamma-up}
    \int_{\Gamma^-} \, \frac{f(z) \, e^{2 i \lambda z (m-k)} }{z^{2m}} \, \text{d}z \,   \, = \, 2 \pi i \, \text{Res} \left[\frac{f(z) \, e^{2 i \lambda z (m-k)} }{z^{2m}},z=0 \right].
\end{equation}
Thus, the integrals in the second term of equation \eqref{eqn:IntExpandGen} become
\begin{align}
    \int_{\Gamma^+} \, \frac{f(z) \, e^{2 i \lambda z (m-k)}}{z^{2m}} \, \text{d}z \, + \, \int_{\Gamma^-}  \, \frac{f(z) \, e^{2 i \lambda z (m-k)} }{z^{2m}} \, \text{d}z \, =  \, 2 \pi i \, \text{Res} \left[\frac{f(z) \, e^{2 i \lambda z (m-k)} }{z^{2m}},z=0 \right].
\end{align}
By applying the Leibniz rule for differentiation to compute the derivatives, we determine the residues and obtain
\begin{align} \label{eqn:Residue1}
    \text{Res} \left[ \frac{f(z) \, e^{2 i \lambda z (m-k)}}{z^{2m}} \,  , z=0 \right] & \, =  \, \, \sum_{s=0}^{2m-1} \, \frac{\, f^{(2m-1-s)} (0)}{s! \, (2m-1-s)!} \, (2 i \lambda (m-k))^s. 
\end{align}
Finally, we get now the following values for the second term of equation \eqref{eqn:IntExpandGen}
\begin{align}\label{eqn:UpperHPfinal}
    \int_{\Gamma^+} \, \frac{f(z) \, e^{2 i \lambda z (m-k)}}{z^{2m}} \, \text{d}z \, + \, \int_{\Gamma^-}  \, \frac{f(z) \, e^{2 i \lambda z (m-k)} }{z^{2m}} \, \text{d}z \, = \, 2 \pi i \, \sum_{s=0}^{2m-1} \, \frac{f^{(2m-1-s)} (0)}{s! \, (2m-1-s)!} \, (2 i \lambda (m-k))^s .
\end{align}

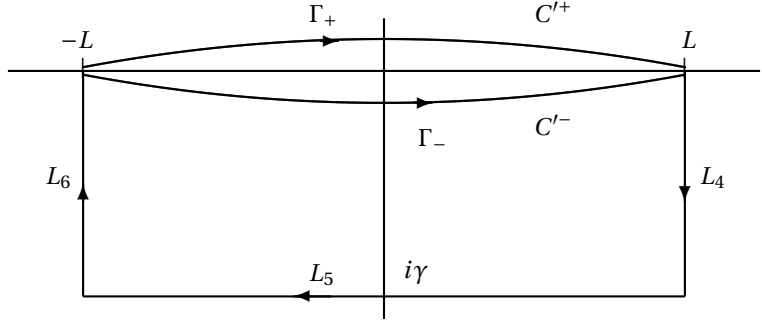
\begin{figure}
\vspace{-11em}
\begin{center}
\setlength{\unitlength}{1cm}
\begin{picture}(12,8)
  \thicklines
  \put(1,4){\line(1,0){10}}    
  \put(6,0.7){\line(0,1){4}}    

  \put(2,1){\line(1,0){8}}     
  \put(2,4){\line(0,-1){3}}     
  \put(10,4){\line(0,-1){3}}    

  \qbezier(2,4.05)(6,4.8)(10,4.05)
  \qbezier(2,3.95)(6,3.2)(10,3.95)

  \put(5,4.65){$\Gamma_+$}
  \put(6.45,3){$\Gamma_-$}

  \put(8,4.65){$C'^+$}
  \put(8,3.15){$C'^-$}

  \put(1.95,3.95){$|$}
  \put(9.95,3.95){$|$}
  \put(1.7,4.3){$-L$}
  \put(9.95,4.3){$L$}
  
  \put(6.25,1.25){$i \gamma$}

  \put(10.2,2.5){$L_4$}  
  \put(5,1.2){$L_5$}   
  \put(1.5,2.5){$L_6$}   

  \put(5.15,4.4){\vector(1,0){.25}}  
  \put(6.4,3.58){\vector(1,0){.25}}  
  \put(10,2.5){\vector(0,-1){0.25}}  
  \put(5.3,1){\vector(-1,0){0.5}}  
  \put(2,2.25){\vector(0,1){.25}}  

\end{picture}
\end{center}
\vspace{-1cm}
\caption{The contour $C'^+$ consists of the path $ \Gamma_+ $, $ L_4 $, $ L_5 $ and $ L_6$, while the contour $C'^-$ consists of $ \Gamma_- $, $ L_4 $, $ L_5 $ and $ L_6$.}
\label{fig:Figure2}
\end{figure}
Let us now evaluate the integrals in the third term of equation \eqref{eqn:IntExpandGen} by  performing a contour integration in the lower half-plane. We consider a contour $C'^+$ as shown in Figure \ref{fig:Figure2} where 
\begin{align}
    \oint_{C'^+} \, \frac{f(z) \, e^{- 2 i \lambda z (k-m)}}{z^{2m}}  \, \text{d}z \, & \, = \, \int_{\Gamma^+} \, \frac{f(z) \, e^{- 2 i \lambda z (k-m)}}{z^{2m}}  \, \text{d}z \, + \int_{L_4} \, 
    \frac{f(L - i y) \, e^{2 i \lambda (L -i y) (m-k)}}{(L - i y)^{2m}}  \, \text{d}y \, \nonumber \\
    & \, \quad + \, \int_{L_5}  \,  \frac{f(x - i \gamma) \, e^{- 2 i \lambda (x - i \gamma) (k-m)}}{(x + i y)^{2m}} \, \text{d}x \, +\int_{L_6} \, 
    \frac{f(-L - i y) \, e^{2 i \lambda (-L -i y) (m-k)}}{(-L - i y)^{2m}}  \, \text{d}y \nonumber \\
    & \,= - 2 \pi i \, \text{Res}\left[\frac{f(z) e^{- 2 i \lambda z (k-m)}}{z^{2m}},z=0 \right].
\end{align}
Since the contour $C'^+$ encloses the pole of the integrand at $z=0$, the integral evaluates to the residue at this pole.
Using the same estimate for $L_1$, the contribution along path $L_4$ vanishes as $L\to\infty$ \cite{boas1954entire}. Similarly, the contribution along the path $L_6$ also vanishes as $L\to -\infty   $\cite{boas1954entire,schmeisser2007approximation}. We now estimate the contribution along the path $L_5$, where
\begin{equation}
    \bigg| \int_{L}^{-L} \, f(x - i \gamma) \, \frac{e^{- 2 i \lambda (x - i \gamma) (k-m)}}{(x + i y)^{2m}} \, \text{d}x \bigg| \,   \, \leq \, \int_{-\infty}^{\infty} \, \big|  f(x - i \gamma) \big| \, \big| \frac{e^{- 2 i \lambda (x - i \gamma) (k-m)}}{(x - i y)^{2m}} \big|  \, \text{d}x \, \leq \, M \,  e^{- \gamma (2 \lambda (k-m) - \tau) }  .\label{eqn:InEqual2}
\end{equation}
Again here, we have used the property of the exponential type function and by the inequality in equation \eqref{eqn:InEqual2}, the contribution along the path $L_5$ vanishes provided 
\begin{align}\label{eqn:condition2}
    \lambda > \frac{\tau}{2(k-m)}, \,\,\,\, \mbox{for} \,\,\,\, k=m+1,\dots,2m,
\end{align}
is satisfied. We then arrive with
\begin{align}\label{eqn:Gamma+down}
    \int_{\Gamma^+} \, \frac{f(z) \, e^{- 2 i \lambda z (k-m)}}{z^{2m}}  \, \text{d}z \,   \, = - \, 2 \pi i \, \text{Res} \left[\frac{f(z) \, e^{- 2 i\lambda z (k-m)}}{z^{2m}} ,z=0 \right].
\end{align}
Also, we consider the contour $C'^-$ in the lower half-plane as shown in Figure \ref{fig:Figure2}, which include the path $\Gamma^-$. We have
\begin{align}
    \oint_{C'^-} \, \frac{f(z) \, e^{- 2 i \lambda z (k-m)}}{z^{2m}}  \, \text{d}z \, & = \, \int_{\Gamma^-} \, \frac{f(z) \, e^{- 2 i \lambda z (k-m)}}{z^{2m}}  \, \text{d}z \, + \int_{L_4} \, 
    \frac{f(L - i y) \, e^{2 i \lambda (L -i y) (m-k)}}{(L - i y)^{2m}}  \, \text{d}y \, \nonumber \\
    & \, \quad + \, \int_{L_5}  \,  \frac{f(x - i \gamma) \, e^{- 2 i \lambda (x - i \gamma) (k-m)}}{(x - i y)^{2m}} \, \text{d}x \, +\int_{L_6} \, 
    \frac{f(-L - i y) \, e^{2 i \lambda (-L -i y) (m-k)}}{(-L - i y)^{2m}}  \, \text{d}y \, = 0,
\end{align}
where no pole is enclosed in this contour. Using the same estimate made for the paths along $L_4, L_5$ and $L_6$ as $L\to\pm\infty $ \cite{boas1954entire,schmeisser2007approximation}, we see that
\begin{align}
    \int_{\Gamma^-} \,  \frac{f(z) \, e^{- 2 i \lambda z (k-m)}}{z^{2m}}  \, \text{d}z \, \, = 0.
\end{align}
Going back to equation \eqref{eqn:Gamma+down} and solving for the residue, we obtain the following result
\begin{align}
    \text{Res} \left[\frac{f(z) \, e^{- 2 i\lambda z (k-m)}}{z^{2m}} ,z=0 \right] & \, = \, \, \sum_{s=0}^{2m-1} \, \frac{ f^{(2m-1-s)} (0)}{s! \, (2m-1-s)!} \, (- 2 i \lambda (k-m))^s .
\end{align}
Thus, the integrals on the third term in equation \eqref{eqn:IntExpandGen}
are equal to 
\begin{align}\label{eqn:LowerHPfinal}
     \int_{\Gamma^+} \, \frac{f(z) \, e^{- 2 i \lambda z (k-m)}}{z^{2m}}  \, \text{d}z + \int_{\Gamma^-} \,  \frac{f(z) \, e^{- 2 i \lambda z (k-m)}}{z^{2m}}  \, \text{d}z = \, - 2 \pi i \, \sum_{s=0}^{2m-1} \, \frac{f^{(2m-1-s)} (0)}{s! \, (2m-1-s)!} \, (- 2 i \lambda (k-m))^s .
\end{align}

Combining the results from equations \eqref{eqn:FPIevenGen}, \eqref{eqn:UpperHPfinal} and \eqref{eqn:LowerHPfinal}, and simplifying the expressions, we arrive with
\begin{align}\label{eqn:I1S}
        I_1 (\lambda) = \, \frac{(2m)!}{2^{2m} \, (m!)^2} \, \;\backslash\!\!\!\!\!\!\!\backslash\!\!\!\!\!\!\int_{0}^{\infty} \, \frac{f(x)}{x^{2m}} \, \text{d}x   \, + & \, \,  \frac{\pi i \, (-1)^m (2m)!}{2^{2m +1}} \, \, \sum_{s=0}^{2m-1} \, \frac{(2 i \lambda )^s}{s! \, (2m-1-s)!}  \, f^{(2m-1-s)} (0) \nonumber \\
        & \, \times \, \left[ \sum_{k=0}^{m-1} \, \frac{(-1)^k \, (m-k)^s }{k! \, (2m-k)!} -\sum_{k=m+1}^{2m} \,\frac{(-1)^k \, (m-k)^s }{k! \, (2m-k)!} 
     \right].
\end{align}
Rewriting the sum over $k=m+1,\dots,2m$ in the last line of equation \eqref{eqn:I1S} as a sum over $k=0,\dots,m-1$, we obtain
\begin{align}\label{eqn:2SumsEven}
    \sum_{k=0}^{m-1} \, \frac{(-1)^k \, (m-k)^s }{k! \, (2m-k)!} -\sum_{k=0}^{m-1} \,\frac{(-1)^k \, (-1)^s  (m-k)^s }{k! \, (2m-k)!} = 2 \sum_{k=0}^{m-1} \, \frac{(-1)^k \, (m-k)^{s}}{k! \, (2m-k)!},
\end{align}
where the left-hand side of equation \eqref{eqn:2SumsEven} is equal to zero when $s$ is even and nonzero when $s$ is odd. Thus, we set $s=2l+1$ in equation \eqref{eqn:I1S}.

Lastly, the least upper bound of the right-hand side of the inequalities \eqref{eqn:condition1} and \eqref{eqn:condition2} is $\tau/2$. Hence, the condition that simultaneously satisfies \eqref{eqn:condition1} and \eqref{eqn:condition2} is $\lambda>\tau/2$.
The proof of Theorem 1 is therefore complete.
\end{proof}

\subsection{Example}
Let us consider the following integral
\begin{equation}\label{eqn:ExampleSin4}
  \, \int_0^{\infty} \, \frac{\sin\sqrt{\tau^2 x^2 + \sigma^2}}{\sqrt{\tau^2 x^2 + \sigma^2}} \, \frac{\sin^4 (\lambda x)}{x^4} \, \text{d}x,
\end{equation}
where $\sin\sqrt{\tau^2 x^2 + \sigma^2}/\sqrt{\tau^2 x^2 + \sigma^2}$ is an entire exponential function of type $\tau$. 
The integral in equation \eqref{eqn:ExampleSin4} is an example of Theorem 1 for $m=2$. Applying the results from Theorem 1, we have
\begin{align}\label{eqn:example1}
        \int_{0}^{\infty} \, \frac{\sin\sqrt{\tau^2 x^2 + \sigma^2}}{\sqrt{\tau^2 x^2 + \sigma^2}} \, \frac{\sin^{4} (\lambda x)}{x^{4}} \, \text{d}x = & \, \, \frac{3}{8} \, \,\, \;\backslash\!\!\!\!\!\!\!\backslash\!\!\!\!\!\!\int_{0}^{\infty} \, \frac{\sin\sqrt{\tau^2 x^2 + \sigma^2}}{x^4 \, \sqrt{\tau^2 x^2 + \sigma^2}}  \, \text{d}x   \, + \,  \frac{3 \, \pi \, \lambda }{24} \, \left( \frac{d^2}{dx^2} \frac{\sin\sqrt{\tau^2 x^2 + \sigma^2}}{\sqrt{\tau^2 x^2 + \sigma^2}} \Bigg|_{x=0} \right) \nonumber \\
        &  \quad + \,  \frac{\pi \, \lambda^3 }{3} \, \left( \frac{\sin\sqrt{\tau^2 x^2 + \sigma^2}}{\sqrt{\tau^2 x^2 + \sigma^2}} \Bigg|_{x=0} \right).
\end{align}
To evaluate the finite part integral in the first term of the right hand side of equation \eqref{eqn:example1}, one may apply the method of Mellin transforms
\cite{galapon2023regularized}. 
The corresponding Mellin transform is given by
\begin{equation}
\mathcal{M}[f;s] = \, \int_0^{\infty} x^{s-1} \,
\frac{\sin\sqrt{\tau^2 x^2 + a^2}}{\sqrt{\tau^2 x^2 + a^2}} \, \text{d}x  \, = \, 2^{(s-3)/2} \tau \sqrt{\pi} \left( \frac{\sigma}{\tau^2}\right)^{(s-1)/2} \, \Gamma \left(s/2\right) \, J_{(1-s)/2} (\sigma),
\end{equation}
for $\text{Re} \, a,b>0$ and $0<\text{Re}(s)<2$ \cite[p.43,2.4.2.30]{brychkov2018handbook}. The right hand side represents the analytic continuation of the Mellin transform beyond its original strip of analyticity, introducing simple poles at $s=0, -2, -4, \dots$. The finite-part integral is obtained by locating the point $s=-3$, which corresponds to the divergent integral. Since $s=-3$ is a regular point in the analytic continuation, its value at this point directly yields the desired finite-part integral,
\begin{equation}\label{eqn:result3}
     \mathcal{M}[f;s=-3] = \;\backslash\!\!\!\!\!\!\!\backslash\!\!\!\!\!\!\int_{0}^{\infty} \, \frac{\sin\sqrt{\tau^2 x^2 + \sigma^2}}{x^4 \, \sqrt{\tau^2 x^2 + \sigma^2}}  \, \text{d}x   =  \frac{\pi \, \tau^3 }{6 \, \sigma^2} \, J_2 (\sigma).
\end{equation}
Using the result in equation \eqref{eqn:result3} and computing for the derivatives in equation \eqref{eqn:example1}, we arrive with
\begin{align}\label{eqn:example1result}
        \int_{0}^{\infty} \, \frac{\sin\sqrt{\tau^2 x^2 + \sigma^2}}{\sqrt{\tau^2 x^2 + \sigma^2}} \, \frac{\sin^{4} (\lambda x)}{x^{4}} \, \text{d}x = & \, \frac{\pi \, \tau^3 }{16 \, \sigma^2} \, J_2 (\sigma)   \, + \,  \frac{\pi \tau^2 \lambda }{8 \sigma^3} \, ( \sigma \cos(\sigma) - \sin (\sigma)) + \,  \frac{\pi \lambda^3 \, }{3 \sigma} \sin (\sigma),\quad \lambda > \tau/2.
\end{align}

In equation \eqref{eqn:example1result}, we have obtained positive orders of $\lambda$ as $\lambda \to \infty$. Further, the behavior of the integrand is illustrated in Figure \ref{fig:Figure3}. We observe that the area under the curve of the integrand increases as $\lambda$ increases. This shows consistency with the polynomial behavior of the obtained expansion which behaves as a polynomial with positive powers of $\lambda$ as $\lambda$ grows large.
\begin{figure}
\centering
\includegraphics[width=0.8\textwidth]{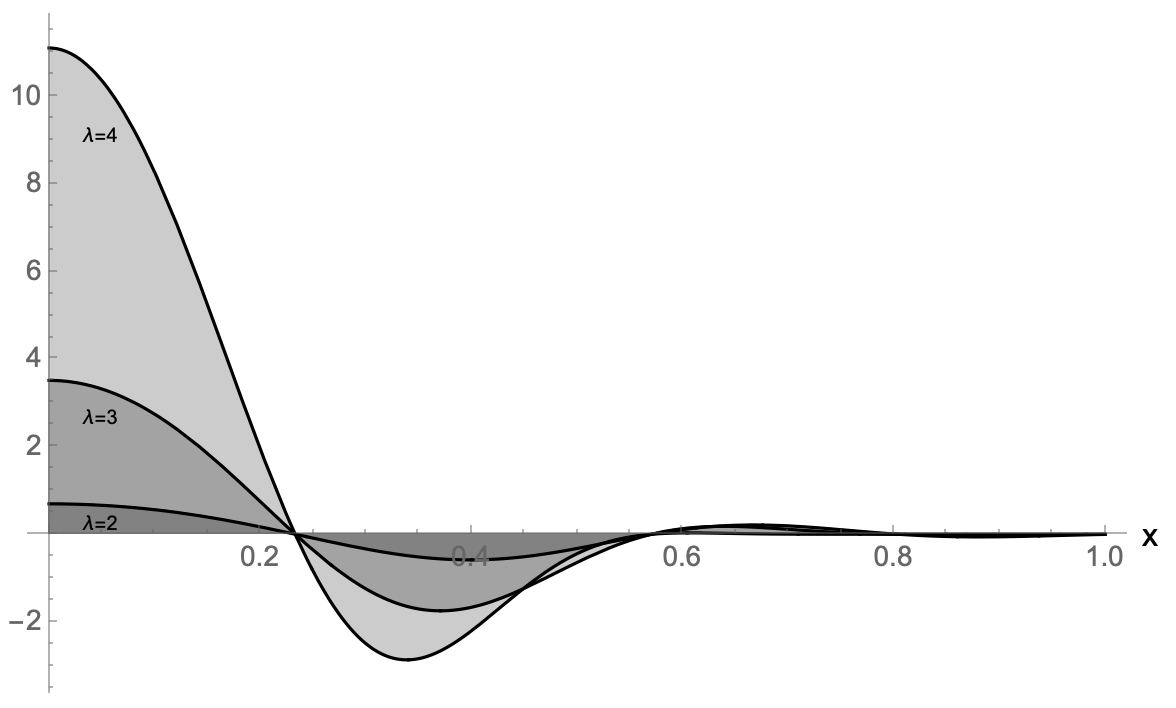}
\caption{The plot of the integrand, $\frac{\sin\sqrt{\tau^2 x^2 + \sigma^2}}{\sqrt{\tau^2 x^2 + \sigma^2}} \, \frac{\sin^{4} (\lambda x)}{x^{4}}$ for $\lambda =2,3,4$.}
\label{fig:Figure3}
\end{figure}

\section{Case 2}
\begin{theorem}
    Let $f(x)$ be an entire function of exponential type $\tau>0$ satisfying $f(-x)=f(x)$. For nonnegative integer $m$, and $\lambda > \tau$, we have
    \begin{align}
        \int_{0}^{\infty} \, \frac{f(x) \, \sin^{2m+1} (\lambda x)}{x^{2m+1}} \, \text{d}x = & \, \frac{(-1)^m \pi  (2m+1)!}{2^{2m+1}} \, \sum_{l=0}^{m} \, D_{l,m}  \, \lambda^{2l},
\end{align}
where
\begin{align}
    D_{l,m} = \frac{(-1)^l \,  f^{(2m-2l)} (0)}{(2l)! \, (2m-2l)!} \, \sum_{k=0}^{m} \, \frac{(-1)^k \, (2m-2k+1)^{2l}}{k! \, (2m+1-k)!}.
\end{align}
\end{theorem}

\begin{proof}
The proof of Theorem 2 is identical to that of Theorem 1.
Let
\begin{equation}\label{eqn:example2}
    I_2 (\lambda) = \int_{0}^{\infty} \, \frac{f(x) \, \sin^{2m+1} (\lambda x)}{x^{2m+1}} \, \text{d}x.
\end{equation}
When $f(-x)=f(x)$, the integrand of $I_2 (\lambda)$ becomes even in $x$ so we can write the integral in equation \eqref{eqn:example2} as
\begin{align}\label{eqn:IntSinc2m+1}
    \int_{0}^{\infty} \, \frac{f(x) \, \sin^{2m+1} (\lambda x)}{x^{2m+1}} \, \text{d}x \, = \, \frac{1}{2} \, \int_{-\infty}^{\infty} \, \frac{f(x) \, \sin^{2m+1} (\lambda x)}{x^{2m+1}} \, \text{d}x.
\end{align}
Here, we proceed with performing the contour integration by expressing the integral on the right hand side of equation \eqref{eqn:IntSinc2m+1} as the average of two contour integrals \cite{galapon2016cauchy}
\begin{equation}\label{eqn:Gamma+-2m+1}
    \int_{-\infty}^{\infty} \, \frac{f(x) \, \sin^{2m+1} (\lambda x)}{x^{2m+1}} \, \text{d}x \, =  \, \frac{1}{2} \left( \int_{\Gamma^+} \, \frac{f(z) \, \sin^{2m+1} (\lambda z)}{z^{2m+1}} \, \text{d}z + \, \int_{\Gamma^-} \, \frac{f(z) \, \sin^{2m+1} (\lambda z)}{z^{2m+1}} \, \text{d}z \right),
\end{equation}
where $\Gamma^+$ and $\Gamma^-$ follow the same path as shown in Figure \ref{fig:Figure1}. The sine function is then expanded in series representation leading to
\begin{align}
    \sin^{2m+1} (\lambda z) \, = \, & \, \frac{i \, (-1)^m (2m+1)!}{2^{2m+1}} \, \sum_{k=0}^{2m+1} \,\frac{(-1)^k}{k! \, (2m+1-k)!} \ e^{i \lambda z (2m+1-2k)}.
\end{align}
Inserting this expansion into the integrals on the right-hand side in equation \eqref{eqn:Gamma+-2m+1} and interchanging the order of summation and integration, we obtain
\begin{align}\label{eqn:OddSum}
    I_2 (\lambda) \, & = \, \frac{i (-1)^m (2m+1)!}{2^{2m+3}} \, \sum_{k=0}^{2m+1} \,\frac{(-1)^k }{k! \, (2m+1-k)!} \, \left[ \int_{\Gamma^+} \, \frac{f(z) \, e^{i \lambda z (2m+1-2k)}}{z^{2m+1}} \, \text{d}z + \int_{\Gamma^-} \, \frac{f(z) \, e^{i \lambda z (2m+1-2k)}}{z^{2m+1}} \, \text{d}z  \right].
    \end{align}
In the summation on the right-hand side of equation \eqref{eqn:OddSum}, we separate it into two sums: one over $k=0,\dots,m$ and another over $k=m+1,\dots,2m+1$. Doing so reveals that the powers of the exponential function in first sum over $k=0,\dots,m$ remain positive, while in the second sum over $k=m+1,\dots,2m+1$, the powers become negative. As a result, we obtain
\begin{align}\label{eqn:IntExpandOdd}
    I_2 (\lambda) \, & = \, \frac{i (-1)^m (2m+1)!}{2^{2m+3}} \, \sum_{k=0}^{m} \,\frac{(-1)^k }{k! \, (2m+1-k)!} \, \left[ \int_{\Gamma^+} \, \frac{f(z) \, e^{i \lambda z (2m+1-2k)}}{z^{2m+1}} \, \text{d}z + \int_{\Gamma^-} \, \frac{f(z) \, e^{i \lambda z (2m+1-2k)}}{z^{2m+1}} \, \text{d}z  \right] \nonumber \\
    & \quad + \, \frac{i (-1)^m (2m+1)!}{2^{2m+3}} \, \sum_{k=m+1}^{2m+1} \,\frac{(-1)^k }{k! \, (2m+1-k)!} \, \left[ \int_{\Gamma^+} \, \frac{f(z) \, e^{- i \lambda z (2k-2m-1)}}{z^{2m+1}} \, \text{d}z + \int_{\Gamma^-} \, \frac{f(z) \, e^{- i \lambda z (2k-2m-1)}}{z^{2m+1}} \, \text{d}z  \right].
\end{align}

We now evaluate the integrals in the first term of equation \eqref{eqn:IntExpandOdd}. Notice that these integrals are identical to those in the second term of equation \eqref{eqn:IntExpandGen}. Following the same contour integration procedure, we analyze the estimates of each integral along the corresponding paths in the contour.
In the upper half plane, we once again consider the contour $C^+$ in Figure \ref{fig:Figure1} and note that no poles are enclosed within this contour.
Following the same estimates along the path $L_1$ and $L_3$, the contributions from these paths vanish in the limit as $L\to \pm \infty$, respectively \cite{boas1954entire,schmeisser2007approximation}. Determining the bound for the integral along the path $L_2$, we have
\begin{equation}
    \bigg| \int_{L}^{-L} \, \frac{f(x + i \gamma) \, e^{i \lambda (x + i \gamma) (2m+1-2k)}}{(x + i \gamma)^{2m+1}} \, \text{d}x \bigg| \, \leq \, \int_{-\infty}^{\infty} \, \big|  f(x + i \gamma) \big| \, \big| \frac{e^{i \lambda (x + i \gamma) (2m+1-2k)}}{(x + i \gamma)^{2m+1}} \big|  \, \text{d}x \, \leq \, M \,  e^{- \gamma (\lambda (2m+1-2k) - \tau) }  .
\end{equation}
Here, we have used the property of the exponential type function for $f(x)$. By imposing the condition
\begin{align}\label{eqn:condition3}
    \lambda>\frac{\tau}{(2m+1-2k)},\,\,\,\, \mbox{for} \,\,\,\, k=0,\dots,m,
\end{align}
the contribution along the path $L_2$ vanishes in the limit as $L\to \infty$. Consequently, we get
\begin{align}
     \int_{\Gamma^+} \frac{f(z) \, e^{i \lambda z (2m+1-2k)}}{z^{2m+1}} \, \text{d}z  = 0.
\end{align}

We also consider the contour $C^-$ in the upper half-plane, which now encloses the pole at $z=0$ of the integrand as shown in Figure \ref{fig:Figure1}. Following a similar contour integration procedure from Theorem 1, and applying the same estimates along the paths $L_1, L_3$ and $L_2$ where we note that their contributions vanish in the limit as $L \to \pm\infty$ and $\lambda>\tau/(2m+1-2k)$ is satisfied, we arrive with
\begin{align}
    \int_{\Gamma^-} \, \frac{f(z) \, e^{i \lambda z (2m+1-2k)}}{z^{2m+1}} \, \text{d}z = 2 \pi i \, \text{Res} \left[\frac{f(z) \, e^{i \lambda z (2m+1-2k)}}{z^{2m+1}},z=0 \right].
\end{align}
From the above results, we proceed by calculating the residues of the enclosed pole to finally obtain
\begin{align}
    \int_{\Gamma^+} \, \frac{f(z) \, e^{i \lambda z (2m+1-2k)}}{z^{2m+1}} \, \text{d}z + \int_{\Gamma^-} \, \frac{f(z) \, e^{i \lambda z (2m+1-2k)}}{z^{2m+1}} \, \text{d}z \, = 2 \pi i \, \sum_{s=0}^{2m} \, \frac{(i \lambda)^s \,  f^{(2m-s)} (0)}{s! \, (2m-s)!} \, (2m-2k+1)^s. \label{eqn:OddGamma1}
\end{align}

We now evaluate the integrals in the second term of equation \eqref{eqn:IntExpandOdd}, which are identical to those in the third term of equation \eqref{eqn:IntExpandGen}.
In the lower half plane, we consider the contour $C'^+$ as shown in Figure \ref{fig:Figure2} where the pole at $z=0$ is enclosed within this path. 
The same estimate is applied to the contributions along the path $L_4$ and $L_6$ in which these integrals vanish in the limit as $L\to \pm\infty$, respectively \cite{boas1954entire,schmeisser2007approximation}. For the integral along the path $L_5$, we determine its bound
\begin{equation}
    \bigg| \int_{L}^{-L} \, f(x - i \gamma) \, \frac{e^{-i \lambda (x + i \gamma) (2k-2m-1)}}{(x + i \gamma)^{2m+1}} \, \text{d}x \bigg| \,   \, \leq \, \int_{-\infty}^{\infty} \, \big|  f(x - i \gamma) \big| \, \big| \frac{e^{-i \lambda (x - i \gamma) (2k-2m-1)}}{(x - i \gamma)^{2m+1}} \big|  \, \text{d}x \, \leq \, M \,  e^{- \gamma (\lambda (2k-2m-1) - \tau) },
\end{equation}
where we have used the property of the exponential type function for $f(x)$. Imposing the condition
\begin{align}\label{eqn:condition4}
    \lambda > \frac{\tau}{(2k-2m-1)}, \,\,\,\, \mbox{for} \,\,\,\, k=m+1,\dots,2m+1,
\end{align}
the integral along the path $L_5$ vanishes in the limit as $L\to\infty$. 
Also, obtaining the value of the residue at $z=0$, we compute the necessary derivatives of the integrand to arrive with
\begin{align}
     \int_{\Gamma^+} \, \frac{f(z) \, e^{- i \lambda z (2k-2m-1)}}{z^{2m+1}} \, \text{d}z \, = - 2 \pi i \, \sum_{s=0}^{2m} \, \frac{(- i \lambda)^s \,  f^{(2m-s)} (0)}{s! \, (2m-s)!} \, (2k-2m-1)^s. \label{eqn:IntGamma2a}
\end{align}

Lastly, we consider the contour $C'^-$ which encloses no poles of the integrand. 
In the limit as $L\to \pm \infty$, we know that the contributions along the paths $L_4$ and $L_6$ vanish respectively \cite{boas1954entire,schmeisser2007approximation}. Additionally, the contribution along the path $L_5$ also vanishes in the limit as $L\to\infty$ provided the condition $\lambda > \tau/ (2k-2m-1)$ for values of $k=m+1,\dots,2m+1$ is satisfied. As a result, we get
\begin{align}
    \int_{\Gamma^-} \, \frac{f(z) \, e^{- i \lambda z (2k-2m-1)}}{z^{2m+1}}  \, \text{d}z \, = 0.
\end{align}
Therefore, we arrive with
\begin{align}
     \int_{\Gamma^+} \, \frac{f(z) \, e^{- i \lambda z (2k-2m-1)}}{z^{2m+1}} \, \text{d}z \, + \int_{\Gamma^-} \, \frac{f(z) \, e^{- i \lambda z (2k-2m-1)}}{z^{2m+1}}  \, \text{d}z \, = - 2 \pi i \, \sum_{s=0}^{2m} \, \frac{(- i \lambda)^s \,  f^{(2m-s)} (0)}{s! \, (2m-s)!} \, (2k-2m-1)^s. \label{eqn:IntGamma2b}
\end{align}

Using the results obtained from equation \eqref{eqn:IntGamma2a} and \eqref{eqn:IntGamma2b}, and substituting them into equation \eqref{eqn:IntExpandOdd}, we get
\begin{align}
    I_2 (\lambda) \, & = \, \frac{(-1)^m \pi (2m+1)!}{2^{2m+2}} \, \sum_{s=0}^{2m} \, \frac{(i \lambda)^s \,  f^{(2m-s)} (0)}{s! \, (2m-s)!} \, \left[ \sum_{k=0}^{m} \,\frac{(-1)^k \, (2m-2k+1)^s}{k! \, (2m+1-k)!}  - \sum_{k=m+1}^{2m+1} \,\frac{(-1)^k \, (2k-2m-1)^s}{k! \, (2m+1-k)!} \right]. \label{eqn:IntExpandOdd2}
\end{align}
For the sum over $k=m+1,\dots,2m+1$ on the right-hand side of equation \eqref{eqn:IntExpandOdd2}, we can rewrite it as a sum over $k=0,\dots,m$. With this, we obtain
\begin{align}\label{eqn:2SumsOdd}
    \sum_{k=0}^{m} \,\frac{(-1)^k \, (2m-2k+1)^s}{k! \, (2m+1-k)!}  + \sum_{k=0}^{m} \,\frac{(-1)^{k} \, (-1)^s \, (2k-2m-1)^s}{k! \, (2m+1-k)!} \, = \, 2 \, \sum_{k=0}^{m} \,\frac{(-1)^k \, (2m-2k+1)^s}{k! \, (2m+1-k)!},
\end{align}
where the left-hand side of equation \eqref{eqn:2SumsOdd} is equal to zero when $s$ is odd and nonzero when $s$ is even so we set $s=2l$ in equation \eqref{eqn:IntExpandOdd2}.

Finally, the least upper bound of the right-hand side of conditions \eqref{eqn:condition3} and \eqref{eqn:condition4} is $\tau$. Thus, when $\lambda>\tau$, the conditions \eqref{eqn:condition3} and \eqref{eqn:condition4} are simultaneously satisfied.
This completes the proof of Theorem 2.

\end{proof}

\subsection{Example}
Let us consider the following integral
\begin{equation}\label{eqn:ExampleSin3}
     \, \int_0^{\infty} \, \frac{\sin\sqrt{\tau^2 x^2 + \sigma^2}}{\sqrt{\tau^2 x^2 + \sigma^2}} \, \frac{\sin^3 (\lambda x)}{x^3} \, \text{d}x,
\end{equation}
where the same entire exponential type function is involved as in equation \eqref{eqn:ExampleSin4}. Applying the results from Theorem 2 where $m=1$, we get
\begin{equation}
     \, \int_0^{\infty} \, \frac{\sin\sqrt{\tau^2 x^2 + \sigma^2}}{\sqrt{\tau^2 x^2 + \sigma^2}} \, \frac{\sin^3 (\lambda x)}{x^3} \, \text{d}x = \frac{\pi \tau^2 }{8 \sigma^3} \, ( \sigma \cos(\sigma) - \sin (\sigma)) + \,  \frac{3 \pi \lambda^2 \, }{8 \sigma} \sin (\sigma), \quad \lambda > \tau.
\end{equation}

A plot of the integrand of equation \eqref{eqn:ExampleSin3} is also illustrated in Figure \ref{fig:Figure4}. We see there that the net area under the curve of the integral in \eqref{eqn:ExampleSin3} increases for increasing $\lambda$ which is consistent with the polynomial behavior of the obtained expansion for positive powers of $\lambda$ as $\lambda\to\infty$. This is contrary to the usual behavior of asymptotic expansions given by the Poincar{\'e} asymptotic expansion which is in inverse powers of $\lambda$ for arbitrary large $\lambda$.
\begin{figure}
\centering
\includegraphics[width=0.8\textwidth]{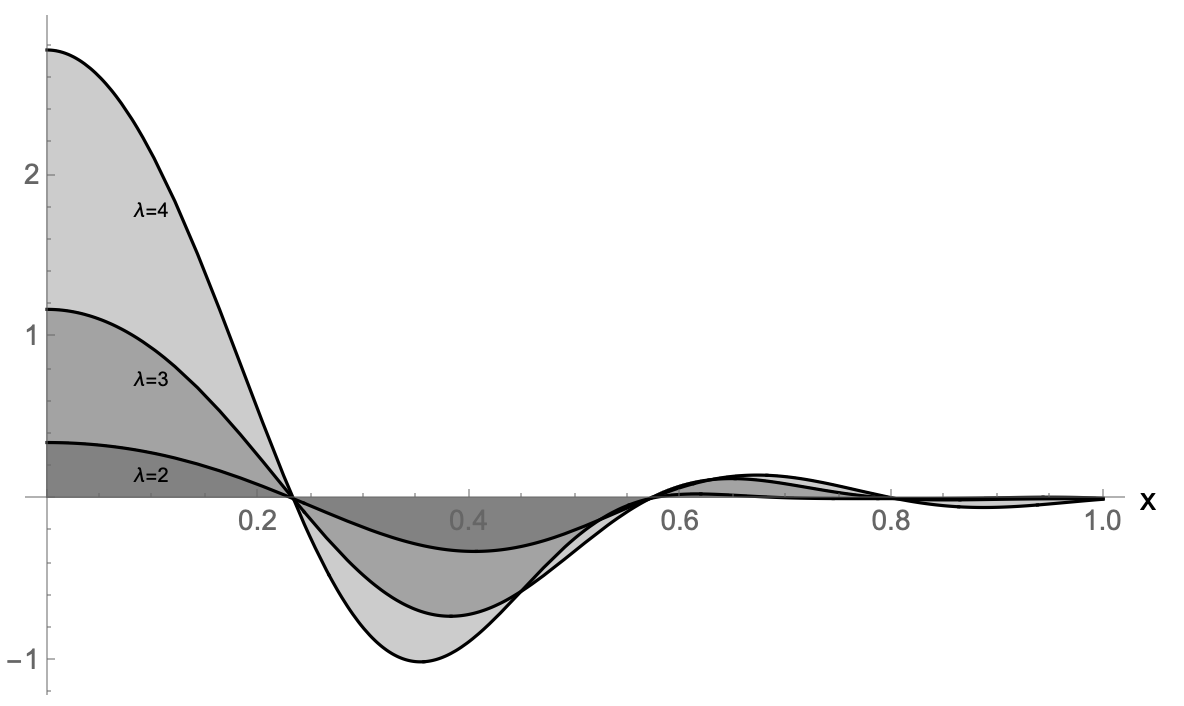}
\caption{The plot of the integrand, $\frac{\sin\sqrt{\tau^2 x^2 + \sigma^2}}{\sqrt{\tau^2 x^2 + \sigma^2}} \, \frac{\sin^{3} (\lambda x)}{x^{3}}$ for $\lambda =2,3,4$.}
\label{fig:Figure4}
\end{figure}

\end{document}